\documentclass[]{birkjour}

\usepackage{amsthm,array}
\usepackage{amsmath}
\usepackage{arydshln}
\usepackage{cite}
\usepackage[colorlinks]{hyperref}
\usepackage[nameinlink]{cleveref}

\usepackage{tikz}
\usetikzlibrary{patterns}

\newtheorem{theorem}{\rm\bf Theorem}[section]
\newtheorem{proposition}[theorem]{\rm\bf Proposition}
\newtheorem{lemma}[theorem]{\rm\bf Lemma}

\newtheorem*{theorem*}{Theorem}
\newtheorem*{theorem 1}{\rm\bf Proposition 1}
\newtheorem*{theorem 2}{\rm\bf Proposition 2}
\newtheorem*{conj*}{Conjecture}

\theoremstyle{definition}
\newtheorem{definition}[theorem]{\rm\bf Definition}

\theoremstyle{remark}
\newtheorem{remark}[theorem]{\rm\bf Remark}

\newtheorem{example}[theorem]{\rm\bf Example}

\def\half#1#2{\begin{matrix}\frac{#1}{#2}\end{matrix}}

\def\R#1{\mathbb{R}^{#1}}

\def\field{K}

 \DeclareMathOperator{\Span}{Span}
 
\DeclareMathOperator{\Idm}{Idm}
\DeclareMathOperator{\Nil}{Nil_2}

\DeclareMathOperator{\card}{card}
\DeclareMathOperator{\Affin}{Affin}
\DeclareMathOperator{\sgn}{sgn}

\begin{document}

%
%
%
%
%
%
%
%

\title[Idempotent geometry in generic algebras]{Idempotent geometry in generic algebras}

\author{Yakov Krasnov}
\address{Department of Mathematics, Bar-Ilan University, Ramat-Gan, 52900, Israel}
\email{krasnov@math.biu.ac.il}
\author{Vladimir G. Tkachev}
\address{Department of Mathematics, Link\"oping University, Link\"oping, 58183, Sweden}
\email{vladimir.tkatjev@liu.se}

\begin{abstract}
Using the syzygy method, established in our earlier paper, we characterize the combinatorial stratification of the variety of two-dimensional real generic algebras. We show that there exist exactly three different homotopic types of such algebras and relate this result to potential applications and known facts from qualitative theory of quadratic ODEs. The genericity condition is crucial. For example, the idempotent geometry in Clifford algebras or Jordan algebras of Clifford type is very different: such algebras always contain nontrivial submanifolds of idempotents.
\end{abstract}

\maketitle



\section{Introduction}

The aim of this paper is to further develop the geometrical approach to generic  commutative (maybe nonassociative) algebras initiated in \cite{KT2018a} and based on the topological index theory and singularity theory. More precisely, we are interested in the following question: How the geometry of idempotents in a generic algebra is determined by its algebraic structure, and vice versa. It is classically known that idempotents  play a distinguished role in associative  (matrix, Clifford) and nonassociative (Jordan, octonions) algebra structures \cite{Porteous}, \cite{McCrbook}; see also very recent results for the so-called axial algebras \cite{HallRS} and nonassociative algebras of cubic minimal cones \cite{NTVbook}, \cite{Tk18a}.

A key ingredient of our approach in \cite{KT2018a} is the Euler-Jacobi formula which gives an algebraic relation between the critical points of a polynomial map and their indices. Indeed,  there is a natural bijection between fixed points of the  squaring  map $\psi:x\to x^2$ in a nonassociative algebra $A$ and its idempotents. By  Bez\'out's theorem, the complexification $A_\mathbb{C}$ has either $2^{\dim A}$ distinct idempotents (counting properly multiplicities and nilpotents which can be thought of as idempotents at infinity), or there are infinitely many idempotents. In the first case, the algebra $A$ is called generic. Then the dichotomy can be essentially reformulated as follows: an algebra is generic if and only if $\half12$ is \textit{not} in the Peirce spectrum of the algebra. The Euler-Jacobi formula applied to the squaring map $\psi$  yields several obstructions both on the spectrum of the idempotents and also the idempotent configurations, the so-called syzygies. In particular, information about the spectrum of sufficiently many idempotents completely determines the rest part of the spectrum and also prescribes possible idempotent constellations.

In this paper, we focus on the vector syzygies, i.e. the obstruction on the possible `geometry' or configurations of idempotents. The main results are presented for two-dimensional real generic algebras where we are able to obtain a complete characterization. The reality assumption is needed, in particular, because we are interested in certain topological invariants. We also relate our results to  potential applications and known facts from qualitative theory of quadratic ODEs.

\section{Preliminaries}

We begin by recalling some standard concepts and definitions of nonassociative algebra. By $A$ we denote a finite dimensional commutative nonassociative algebra over a subfield $\field$ of the field of complex numbers.  The algebra multiplication is denoted by juxtaposition. Thus, $xy=yx$ for all $x,y\in A$ but $x(yz)$ and $(xy)z$ are not necessarily equal. By a slight abuse of terminology, an algebra always mean a commutative nonassociative algebra.

If not explicitly stated otherwise, we shall assume that $\field=\R{}$, the field of real numbers. By $A_\mathbb{C}$ we denote the complexification of $A$ obtained in an obvious way by extending of the ground field such that $\dim A=\dim_\mathbb{C} A_\mathbb{C}$.

An operator of multiplication by $y\in A$ is denoted by $L_y,$ i.e. $L_yx=yx$. An element $c$ is called idempotent if $c^2=c$ and 2-nilpotent if $c^2=0$. By $\Idm(A)=\{0\ne c\in A:c^2=c\}$ we denote the set of all \textit{nonzero}  idempotents of $A$ and the complete set of idempotents will be  denoted by
$$
\Idm_0(A)=\{0\}\cup \Idm(A).
$$
By $\Nil(A)$ we denote the set of $2$-nilpotents, i.e. the elements $x$ of $A$ such that $x^2=0$.

Any semisimple idempotent $0\ne c=c^2\in A$ induces the corresponding  Peirce decomposition:
$$
A=\bigoplus_{\lambda\in\sigma(c)}A_c(\lambda),
$$
where $cx=xc=\lambda x$ for any $x\in A_c(\lambda)$ and $\sigma(c)$ is the Peirce spectrum of $c$ (i.e. the spectrum of the operator $L_c$).
The Peirce spectrum $\sigma(A)=\{\lambda_1,\ldots, \lambda_s\}$ of the algebra $A$ is the set of all possible distinct eigenvalues $\lambda_i$ in $\sigma(c)$, when $c$ runs  all idempotents of $A$.

Given a subset $E\subset A$ we denote by $\Span(E)$ the subspace of $A$ spanned by $A$. By $\Affin(E)$ we denote the affine span  of $E$, i.e. the smallest affine subspace of $A$ that contains $E$, regarded as a set of \textit{points} in $A$. Equivalently,
$$
\Affin(E)=e+\Span(E-e),
$$
where $e\in E$ is an arbitrary element and $E-e=\{x-e:x\in E\}$.

\section{General facts on the geometry of idempotents}
Recall \cite{KT2018a} that an nonassociative algebra over $\field$ is called \textit{generic} if its complexification $A_{\mathbb{C}}$ contains exactly $2^n$ distinct idempotents.  Then the following characterization holds:

\begin{proposition}[Theorem~3.2 in \cite{KT2018a}]\label{pro:KT}
If $A$ is a commutative generic algebra  then $\half12\not\in\sigma(A)$. In the converse direction: if $\half12\not\in\sigma(A)$ and $A$ does not contain $2$-nilpotents then $A$ is generic.
\end{proposition}

Let us make some elementary observations. First, note that  \textit{any two nonzero idempotents are linearly independent}. Indeed, $c_2=\lambda c_1$ implies $c_2=\lambda^2c_1$, therefore, since $\lambda\ne0$ we have $\lambda=1.$ Furthermore, we have

\begin{proposition}\label{pro:line}
If two nonzero idempotents lie on a line $\ell\subset A$. Then either of the following cases happens:
\begin{itemize}
\item[(a)]
the whole line consists of idempotents: $\ell\subset \Idm(A)$, $\half12\in \sigma(c)$ for any $c\in \ell$, and for any two points $c_1,c_2\in\ell$ there holds $c_1-c_2\in \Nil(A)$;
\item[(b)]
there are only two idempotents on $\ell$.
\end{itemize}
\end{proposition}

\begin{proof}
Let $\ell$ be a line in $A$ and let $c_1,c_2\in \ell$, where $c_i\in \Idm(A)$. Suppose (b) does not hold, i.e. there exists another idempotent on the line: $c\in \ell\cap \Idm(A)$ and $c\ne c_1,c_2$. Then there exists $\alpha\in \R{}$, $\alpha(\alpha-1)\ne0$ such that
\begin{equation}\label{alpha}
(\alpha c_1+(1-\alpha) c_2)^2=\alpha c_1+(1-\alpha) c_2,
\end{equation}
therefore $\alpha(1-\alpha)(c_1+c_2-2c_1c_2)=0$, implying by the made assumption that $c_1+c_2=2c_1c_2$, or equivalently $(c_1-c_2)^2=0$, hence $c_1-c_2\in \Nil(A)$. It also follows that  \eqref{alpha} holds true for all $\alpha \in \field$, therefore $\ell\subset \Idm(A)$. Let $c=\alpha c_1+(1-\alpha) c_2\in \ell$. Then
\begin{align*}
c(c_1-c_2)&=\alpha c_1+(\alpha-1)c_2+(1-2\alpha)c_1c_2\\
&=\alpha c_1+(\alpha-1)c_2+\half12(1-2\alpha)(c_1+c_2)\\
&=\half12(c_1-c_2),
\end{align*}
therefore $c_1-c_2\ne0$ is an eigenvector corresponding to $\half12$. This proves that $\half12\in \sigma(c)$ for any $c\in \ell$. The proposition follows.
\end{proof}

%

\begin{lemma}[Square identity]
\label{lem:square}
Let $A$ be a commutative algebra over $\field$ and let $c_1,\ldots,c_{k+1}$ be idempotents in $A$ such that
\begin{equation}\label{alph0}
c_{k+1}=\sum_{i=1}^{k}\alpha_i c_i, \qquad
\sum_{i=1}^{k}\alpha_i=1.
\end{equation}
Then
\begin{equation}\label{inverse0}
\sum_{i<j}\alpha_i\alpha_j(c_i-c_j)^2=0.
\end{equation}
\end{lemma}

\begin{proof}
Using $c_i^2=c_i$, $1\le i\le k+1$ we obtain
\begin{equation}\label{suu}
\sum_{i=1}^{k}\alpha_i c_i=
\sum_{i=1}^{k}\alpha_i^2 c_i+
2\sum_{i<j}\alpha_i\alpha_j c_ic_j.
\end{equation}
Using the second identity in \eqref{alph0} yields
$$
\alpha_i(1-\alpha_i)=\sum_{j\ne i}\alpha_i\alpha_j,
$$
therefore applying the latter identity to \eqref{suu} readily yields
$$
0=\sum_{i<j}\alpha_i\alpha_j (2c_ic_j-c_i-c_j)
=\sum_{i<j}\alpha_i\alpha_j (c_i-c_j)^2.
$$
which proves \eqref{inverse0}.
\end{proof}

\begin{remark}The geometrical meaning of \eqref{alph0} is clear.
Note that \eqref{alph0} implies
$$
\sum_{i=1}^{k}\alpha_i (c_i-c_{k+1})=0,
$$
i.e. $\{c_1-c_{k+1}, \ldots, c_k-c_{k+1}\}$ lie in a $(k-1)$-dimensional subspace of $A$, equivalently, $c_{k+1}$ lies in the affine subspace $\Affin(c_1, \ldots, c_{k})$. Conversely, if  $c_{k+1}$ lies in the affine subspace $\Affin(c_1, \ldots, c_{k})$ then \eqref{alph0} holds. Then $\alpha_i$ can be thought of as a sort of barycentric coordinates (even if $c_i$ are linearly dependent).
\end{remark}

\begin{proposition}\label{pro:pi2}
Let $A$ be a generic algebra over $\field$. If $\Pi$ is a $2$-dimensional affine subspace of $A$ then there exists at most $4$ distinct idempotents in $\Pi$, i.e.
$$
\card (\Idm(A)\cap \Pi)\le 4.
$$
\end{proposition}

\begin{proof}
Let us suppose by contradiction that there exist $5$ distinct  idempotents $c_i$, $1\le i\le 5$ in $\Pi$. If some of $c_i=0$ then $\Pi$ is a 2-dimensional subspace. Let, for example, $c_5=0$. Since any two nonzero idempotents non-collinear, $c_1,c_2$ is a basis in $\Pi$, hence
$$
c_3=\alpha c_1+\beta c_2,
$$
where $\alpha\beta\ne0$ because $c_3$ is distinct from $c_1,c_2$. This yields
$$
\alpha c_1+\beta c_2=\alpha^2c_1+\beta^2c_2+2\alpha\beta c_1c_2,
$$
hence
$$
c_1c_2=\frac{1}{2\alpha\beta }(\alpha(1-\alpha)c_1+\beta(1-\beta)c_2),
$$
i.e. $\Pi=\Span (c_1,c_2)$ is a subalgebra. Since $\Pi$ contains $5$ distinct idempotents (counting $0$) then the subalgebra $\Pi$ is non-generic, therefore $A$ is non-generic too, a contradiction. Thus, $\Pi\not\ni 0$.

Let $\{c_i,c_j,c_k\}$  be any triple with distinct indices $1\le i,j,k\le 5$. Since $A$ is generic, $c_i,c_j,c_k$ cannot lie on the same line. Also, by the above argument, the affine subspace $\Affin(c_i,c_j,c_k)$ does not contain the origin, thus $\{c_i,c_j,c_k\}$  are linearly independent. For dimension reasons, $\Pi$ is contained in the 3-dimensional subspace $V_{ijk}:=\Span(c_i,c_j,c_k)$ of $A$. In particular, we can write
\begin{align*}
c_4&=\alpha_1c_1+\alpha_2c_2+\alpha_3c_3,\\
c_5&=\beta_1c_1+\beta_2c_2+\beta_3c_3,
\end{align*}
where
$$
\sum_{i=1}^3\alpha_i=\sum_{i=1}^3\beta_i=1\quad \text{ and all $\alpha_i,\beta_i$ are nonzero}
$$
(where the latter condition follows from the fact that no pair of the idempotents can lie on the same line).
Applying Lemma~\ref{lem:square}, we find
\begin{equation}\label{split}
\begin{split}
\frac{1}{\alpha_1}(c_2-c_3)^2+\frac{1}{\alpha_2}(c_1-c_3)^2 +\frac{1}{\alpha_3}(c_1-c_2)^2&=0, \\
\frac{1}{\beta_1}(c_2-c_3)^2+\frac{1}{\beta_2}(c_1-c_3)^2 +\frac{1}{\beta_3}(c_1-c_2)^2&=0,
\end{split}
\end{equation}
Since $\{c_4,c_5,c_k\}$, $1\le k\le 3$, are linearly independent, we have
$$\left|
  \begin{array}{cc}
    \alpha_i & \alpha_j \\
    \beta_i & \beta_j \\
  \end{array}
\right|\ne0
$$
Then it follows from \eqref{split} that
$$
\frac{\alpha_i\beta_i}{\left|
  \begin{array}{cc}
    \alpha_k & \alpha_i \\
    \beta_k & \beta_i \\
  \end{array}
\right|}(c_k-c_i)^2=\frac{\alpha_j\beta_j}{\left|
  \begin{array}{cc}
    \alpha_k & \alpha_j \\
    \beta_k & \beta_j \\
  \end{array}
\right|}(c_k-c_j)^2, \qquad  \{i,j,k\}=\{1,2,3\}.
$$
Therefore $(c_2-c_3)^2, (c_2-c_1)^2$ and $(c_3-c_1)^2$ are collinear, i.e.
$$
(c_1-c_2)^2=\lambda(c_1-c_3)^2,\quad \lambda\in \field.
$$
Expanding the latter identity yields
$$
c_1u=\half12(1-\lambda)c_1+\half12 u, \qquad \text{where }u:=c_2-\lambda c_3,
$$
hence,
\begin{align*}
c_1((\lambda-1)c_1+u)&=(\lambda-1)c_1+\half12(1-\lambda)c_1+\half12 u\\
&=\half12((\lambda-1)c_1+u).
\end{align*}
This shows that $(\lambda-1)c_1+u$ is an eigenvector of $L_{c_1}$ with eigenvalue $\half12$, i.e. $\half12\in \sigma(A)$. This proves by Proposition~\ref{pro:KT} that $A$ is not generic, a contradiction.

\end{proof}

Note however that there are (necessarily non-generic) algebras with $5$ idempotents lying in a two-dimensional affine subspace $\Pi\subset A$. A typical situation is when these 5 idempotents lie on a quadric in $\Pi$, see two examples below.

\begin{example}  Consider an algebra
$A$ with the ground vector space $\mathbb{R}^3$ equipped by the following multiplication rule\footnote{Since all algebras are commutative, the multiplication structure is uniquely determined by the algebra square: $x\to x^2$. Namely, the multiplication then recovered by linearization $xy=\half12((x+y)^2-x^2-y^2)$}:
\begin{equation}
x^2=(x_1,x_2,x_3)^2=(x_1^2-x_2x_3,x_2^2-x_1x_3,x_3^2-x_1x_2).
\end{equation}

Then the set of a nonzero idempotents is the circle
$$
\Idm(A)=\{x_1+x_2+x_3=1,(x_1-\half13)^2+(x_2-\half13)^2+ (x_3-\half13)^2=\half23\}
$$
and the set of a nonzero 2-nilpotents is
$$\Nil(A)=\{x_1^3=x_2^3=x_3^3\}.
$$
In particular, there exist complex (non-real) 2-nilpotents in the complexification $A_\mathbb{C}$. Furthermore, it is straightforward to verify that the spectrum of any $c\in \Idm(A)$ is constant: $\sigma(c)=\{1,\half12,-\half12\}$.
\end{example}

\begin{example}
Let   $b(x,y)$ be a bilinear form on $\R{n}$, $n\ge 1$. Consider an algebra $A$ on $\R{}\times \R{n}$ with the following multiplication rule:
\begin{equation}
(x_0,x)\bullet (y_0,y)=(x_0y_0+b(x,y), \, x_0y+y_0x), \quad x_0,y_0\in\R{},\,\,\, x,y\in \R{n}.
\end{equation}
The algebra is obviously commutative. It is well-known that the algebra $A$ is Jordan, see \cite{FKbook}. It is called a spin-factor \cite{McCrbook} or a Jordan algebra of Clifford type.  Note that $e=(1,0)$ is the algebra unit. The set of nonzero idempotents is given by
$$
\Idm(A)=\{c=(x_0,x): x_0=\half12,\,\, b(x,x)=\half14\}\cup \{e\}
$$
and the set of a 2-nilpotents is $\Nil(A)=\{(x_0,x): x_0=b(x,x)=0\}$. The spectrum of any nonunital idempotent $c\ne e$ is $\sigma(c)=\{1,0,\half12,\ldots,\half12\}$. For example, let us consider $n=2$ and $b(x,y)=x_1y_1+x_2y_2$ the usual Euclidean inner product. Then in the corresponding Jordan algebra,   $\Idm(A)$ contains the circle $\{x_0=\half12,\,\, x_1^2+x_2^2=\half14\}$.
\end{example}

A further analysis of the proof of Propsition~\ref{pro:pi2} combined with Lemma~\ref{lem:square} makes it is plausible to believe that the following conjecture is true:

\begin{conj*}
If $\Pi$ is a $k$-dimensional affine subspace of a generic algebra  then at most $2^k$ distinct idempotents can be contained in $\Pi$.
\end{conj*}


\section{Two-dimensional real generic algebras}
Though the two-dimensional nonassociative algebras is a very well-studied object, see for example H.~Petersson's paper \cite{Petersson00}, our   results below have a different nature.

\begin{definition}
A nonassociative algebra over $\R{}$ is called a \textit{real generic}  algebra if its complexification $A_{\mathbb{C}}$ is generic and additionally $\Idm(A)=\Idm(A_{\mathbb{C}})$.
\end{definition}

In this section we shall assume that $A$ is a real generic algebra and $\dim_{\R{}} A=2$. Let $c_0=0$ and $c_1,c_2,c_3$ denote  four  distinct  idempotents of $A$.

Then since $\dim A=2$ and $1\in\sigma(c_i)$ for $1\le i\le 3$, the second eigenvalue $\lambda_i$ is also real.  Let us associate to each $c_i$ the value of the characteristic polynomial of $L_{c_i}$ evaluated at $\half12:$
\begin{equation}\label{mmu}
\chi_{c_i}(\half12)=(\half12-1)(\half12-\lambda_i) =\half14(2\lambda_i-1)\in \R{},
\end{equation}
where $\sigma(c_i)=\{1,\lambda_i\}$, $1\le i\le3$. Since $A$ is generic, $\lambda_i\ne \half12$, hence $\chi_{c_i}(\half12)\ne 0$.

The sign of $\chi_{c_i}(\half12)$ has a clear topological interpretation. Let us consider the vector field
$$
\psi_A(x)=x^2-x: A\to A.
$$
Any zero $c$ of $\psi_A$ is an idempotent of $A$ and vice versa. If $A$ is generic then all zeros of $\psi_A$ are isolated. Recall that the index of $\psi_A$ at an isolated zero $c$ is the degree of the (normalized) mapping $\psi_A(x)$ nearby $c$ \cite{MilnorDT}.   Let $c\in \Idm(A)$ and let $z\in A$ then
$$
\half{1}{\epsilon}\psi_A(c+\epsilon z)=(2cz-z)+\epsilon z^2=2(L_c-\half12)z+\epsilon z^2.
$$
Since $A$ is generic,
\begin{equation}\label{chi1}
\half{1}{4\epsilon^2}\det d\psi_A(c)=\chi_c(\half12)=\det (L_c-\half12)\ne0,
\end{equation}
therefore, the index
\begin{equation}\label{chi2}
\mathrm{ind}_c\psi_A=\sgn \chi_c(\half12)
\end{equation}

By the Poincar\'{e}-Hopf theorem, the total index of the vector field is equal to the Euler characteristic of $A$, i.e. vanishes. In particular, this yields the index at infinity
$$
\mathrm{ind}_\infty\psi_A=-\sum_{i=1}^3\sgn \chi_c(\half12).
$$
Note that the latter sum is a topological invariant. This in particular means that $\mathrm{ind}_\infty\psi_A$ is invariant under small deformations of a generic algebra.

Our main goal  is to study possible configurations of the idempotent quadruple $(c_0,c_2,c_2,c_3)$ from topological point of view. To this end, note that there is a syzygy for the nontrivial part of the algebra spectrum established explicitly in \cite{KT2018a} (see also \cite{Walcher1}):
\begin{equation}\label{syzygies}
4\lambda_1\lambda_2\lambda_3- \lambda_1-\lambda_2-\lambda_3+1=0.
\end{equation}
Another way to present \eqref{syzygies} in a generic algebra is
\begin{equation}\label{syzygies0}
\sum_{i=1}^3\frac{-1}{4\chi_{c_i}(\half12)} =\frac{1}{1-2\lambda_1}+\frac{1}{1-2\lambda_2}+\frac{1}{1-2\lambda_3}=1,
\end{equation}
see (40) in \cite{KT2018a}.  Furthermore, the  idempotents also are balanced by the following `mean value' type identity:
\begin{equation}\label{baryc}
\sum_{i=1}^3\frac{c_i}{\chi_{c_i}(\half12)}=0.
\end{equation}
By virtue of \eqref{syzygies0} let us rewrite the latter identity as
\begin{equation}\label{baryc1}
\sum_{i=1}^3a_i c_i=0=c_0, \qquad
\sum_{i=1}^3a_i=1=:-a_0,
\end{equation}
where
\begin{equation}\label{aui}
a_i=\frac{-1}{4\chi_{c_i}(\half12)}=\frac{1}{1-2\lambda_i}.
\end{equation}
The vector of coefficients $a:=(a_1,a_2,a_3)\in\R{}$ can be thought as the barycentric coordinates of $c_0=0$ in the affine coordinate system $(c_1,c_2,c_3)$ in $A$. Note that the case $(-,-,-)$ is impossible by virtue of the second relation in \eqref{baryc1}. Therefore  there are only three possible \textit{unordered} sign configurations of $(a_1,a_2,a_3)$:
$$
\begin{array}{rccc}
\text{sign configurations:}\qquad &\vphantom{\int\limits_1^2}(+,+,+) &  (+,+,-) & \quad (+,-,-)\\
&\updownarrow &\updownarrow &\updownarrow \\
\text{types:}\qquad &\vphantom{\int\limits_1^2}(\mathrm{i})&(\mathrm{ii})&(\mathrm{iii})
\end{array}
$$
referred further to as types (i), (ii) and (iii) respectively. These types can be characterized geometrically, as follows: the three lines passing through nonzero idempotents $c_i$ divide the plane $A$  into $7=1+3+3$ disjoint open domains, see Fig.~\ref{fig:idpm2} below. It is straightforward to verify that the sign of the triple $(a_1,a_2,a_3)$ is constant in each domain, see Fig.~\ref{fig:idpm2}.

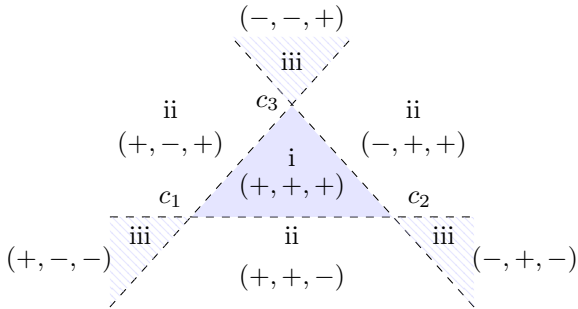
\begin{figure}[ht]
\begin{tikzpicture}[scale=0.08]

\draw[dashed] (-30,0)--(30,0);
\draw[dashed] (-30,-15)--(10,30);
\draw[dashed] (30,-15)--(-10,30);
\fill[pattern=north west lines, pattern color=blue!20!white,opacity=0.7] (10,30) -- (-10,30) -- (0,18.5) -- cycle;
\fill[pattern=north west lines, pattern color=blue!20!white,opacity=0.7] (-30,0) -- (-16.5,0) -- (-30,-15) -- cycle;
\fill[pattern=north west lines, pattern color=blue!20!white,opacity=0.7] (30,0) -- (16.5,0) -- (30,-15) -- cycle;
\fill[blue!20!white,opacity=0.5] (16.5,0) -- (-16.5,0) -- (0,18.5) -- cycle;
\draw (25,-3) node {iii};
\draw[right] (28,-7) node {$(-,+,-)$};
\draw (-25,-3) node {iii};
\draw[left] (-28,-7) node {$(+,-,-)$};
\draw (0,26) node {iii};
\draw[] (0,33) node {$(-,-,+)$};
\draw (0,10) node {i};
\draw (0,5) node {$(+,+,+)$};
\draw (0,-3) node {ii};
\draw (0,-10) node {$(+,+,-)$};
\draw (-20,18) node {ii};
\draw (-20,12) node {$(+,-,+)$};
\draw (20,18) node {ii};
\draw (20,12) node {$(-,+,+)$};
\draw[] (21,3) node {$c_2$};
\draw[] (-20,3) node {$c_1$};
\draw[] (-4,19) node {$c_3$};
\end{tikzpicture}
\caption{The sign of $(a_1,a_2,a_3)$} \label{fig:idpm2}
\end{figure}

\noindent
This yields the following three  \textit{unordered} configurations of sign:

\begin{table}[ht]
\begin{tabular}{cclcl}
$\{+,+,+\}$
&&type \textbf{ (i)} && $c_0$ is inside of the triangle $\Delta(c_1,c_2,c_3)$
\\
\\
$\{+,+,-\}$&&type \textbf{ (ii)} && $c_0$ is in an unbounded trapezius domain
\\
\\
$\{+,-,-\}$&&type \textbf{ (iii)} && $c_0$ is in an unbounded corner
\end{tabular}
\end{table}

\noindent
The sign of $\chi_{c_i}(\half12)$ is an important parameter here. Indeed, as it was remarked above, $\sgn \,\chi_{c_i}(\half12)$ is exactly the index of the vector field $\psi_A(x)=x^2-x$, thus carrying information on the topological `charge' of the corresponding idempotent $c_i$. Note that since $\chi_0(t)=t^2$, we have
$$
\sgn \chi_{c_0}(\half12)=+1\quad \Rightarrow \quad \sgn a_0=-1,
$$
i.e. the sign of the zero idempotent is always negative (the corresponding charge $a_0$ has the opposite sign, i.e. positive).
The three possible configurations of all four idempotents  with their charges are shown in Fig.~\ref{fig:idpm3}.

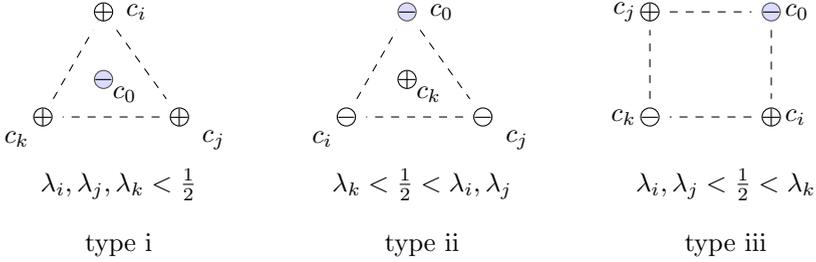
\begin{figure}[ht]
\begin{tikzpicture}[scale=0.1]
\node (ck) at (-47,9){$+$};
\draw (-47,9) circle (12mm);
\draw (-39,23) circle (12mm);
\node (ci) at (-39,23){$+$};
\draw (-29,9) circle (12mm);
\node (cj) at (-29,9){$+$};
\draw[fill=blue!20!white,opacity=0.7] (-39,14) circle (12mm);
\node (c0) at (-39,14){$-$};

\draw[right] (-39,12) node {$c_0$};
\draw[right] (-39,23) node {\, $c_i$};
\draw[right] (-29,6) node {\, $c_j$};
\draw[left] (-47,6) node {$c_k$\, };
\draw [dashed]  (ci)--(cj)--(ck)--(ci); 
\draw (-37,-8) node {type i};
\draw (-37,0) node {$\lambda_i, \lambda_j,\lambda_k<\half12$};

\draw (-47+40,9) circle (12mm);
\node (ak) at (-47+40,9){$-$};
\draw[fill=blue!20!white,opacity=0.7] (-39+40,23) circle (12mm);
\node (ai) at (-39+40,23){$-$};
\draw (-29+40,9) circle (12mm);
\node (aj) at (-29+40,9){$-$};
\draw (-39+40,14) circle (12mm);
\node (a0) at (-39+40,14){$+$};

\draw[right] (-39+40,12) node {$c_k$};
\draw[right] (-39+40,23) node {\, $c_0$};
\draw[right] (-29+40,6) node {\, $c_j$};
\draw[left] (-47+40,6) node {$c_i$\, };
\draw [dashed]  (ai)--(aj)--(ak)--(ai); 
\draw (-37+40,0) node {$\lambda_k<\half12<\lambda_i, \lambda_j$};
\draw (-37+40,-8) node {type ii};

\draw (13+20,9) circle (12mm);
\draw[left] (13+20,9) node {$c_k$\, };
\draw (13+20,23) circle (12mm);
\draw[left] (13+20,23) node {$c_j$\, };
\draw (29+20,9) circle (12mm);
\draw[right] (29+20,23) node {\,$c_0$};
\draw[fill=blue!20!white,opacity=0.7] (29+20,23) circle (12mm);
\draw[right] (29+20,9) node {\,$c_i$};

\draw (13+20,9) node(b0) {$-$};
\draw (13+20,23) node(b1) {$+$};
\draw (29+20,9) node(b2) {$+$};
\draw (29+20,23) node(b3) {$-$};
\draw [dashed]  (b0)--(b1)--(b3)--(b2)--(b0); 
\draw (23+20,0) node {$\lambda_i, \lambda_j<\half12<
\lambda_k$ };
\draw (23+20,-8) node {type iii};
\end{tikzpicture}
\caption{The three configuration types ($c_0$ is marked in grey)} \label{fig:idpm3}
\end{figure}

The three types can be described in terms of the Peirce spectrum as follows. To this end, note that it follows from   \eqref{syzygies} that
\begin{equation}\label{lambda}
\lambda_k-\frac12=
-\frac{(2\lambda_i-1)(2\lambda_j-1)}{4\lambda_i\lambda_j-1},
\end{equation}
where $i,j,k$ is a permutation of $1,2,3$. Therefore if  $\lambda_i,\lambda_j>\half12$ then $\lambda_k<\half12$. In that case, $a_i,a_j<0$ and $a_k>0$, therefore we have type (ii) configuration. Similarly, if $\lambda_i,\lambda_j<\half12$ then $\lambda_k>\half12$, hence $a_i,a_j>0$ and $a_k<0$, therefore we have type (iii) configuration.

Let us show that all three types of generic algebras are realizable. It is natural to ask if the condition \eqref{syzygies} is sufficient to have an algebra with the Peirce spectrum $(\lambda_1,\lambda_2,\lambda_3)$. The following proposition shows that this is the case.

\begin{theorem}
If $A$ is a generic two-dimensional commutative algebra over a field $\field$ then its Peirce spectrum satisfies \eqref{syzygies} and \eqref{baryc}. In the converse direction, given three numbers $\lambda_1,\lambda_2,\lambda_3\in \field$ satisfying \eqref{syzygies}, there exists a generic two-dimensional commutative algebra over  $\field$ such that $\sigma(c_i)=\{1,\lambda_i\}$, $i=1,2,3$, where $c_i$ are three distinct nonzero idempotents.
\end{theorem}

\begin{proof}
For the proof of the first part we refer to Theorem~4.7 in our recent paper \cite{KT2018a}. Thus, let us suppose that $\lambda_1,\lambda_2,\lambda_3\in \field$ satisfy \eqref{syzygies}. We consider a two-dimensional vector space $V$ over $\field$. Let $c_1,c_2$ be a basis of $V$. Turn $V$ into an algebra $A=(V,\circ)$ by setting
$$
c_1\circ c_1=c_1, \quad c_2\circ c_2=c_2, \qquad c_1\circ c_2=\lambda_2 c_1+\lambda_1c_2.
$$
Then $c_1,c_2$ are obviously idempotents in $A$, and it is straightforward to see that $\sigma(c_i)=\{1,\lambda_i\}$, $i=1,2$. Next, let
$$c_3=\alpha_2 c_1+ \alpha_1 c_2,
$$
where
$$
\alpha_i=\frac{2\lambda_i-1}{4\lambda_1\lambda_2-1}, \,\,i=1,2.
$$
Then using the definition of the multiplicative table on $A$ we obtain
\begin{align*}
c_3\circ c_3&=\alpha_2(\alpha_2+2\lambda_2\alpha_1)c_1+
\alpha_1(\alpha_1+2\lambda_1\alpha_2)c_2\\
&=\alpha_2 c_1+ \alpha_1 c_2\\
&=c_3.
\end{align*}
Hence $c_3$ is a nonzero idempotent  distinct of $c_1,c_2$. Furthermore, using \eqref{lambda} we obtain $\sigma(c_3)=\{1,\lambda_3\}$, hence $A$ is a generic algebra with the desired Peirce spectrum.
\end{proof}

Taking into account the latter theorem, we see that all three types are realizable. We give some explicit examples of each type.

\begin{example}[Types (i) and (ii)]\label{ex:typeib}
Let us define a nonassociative algebra $H(\tau)$ be an algebra on $\R{2}$ with multiplication
$$
(x_1,x_2)\circ (y_1,y_2)=(x_1y_1-x_2y_2,\half12(1-\tau^2)(x_1y_2+x_2y_1)),
$$
where $\tau>0$, $\tau\ne 1$, is some fixed real. Then
$$
\Idm(A)=\{c_0=0,\,c_1=(\half1{1-\tau^2},\,\half{-\tau}{1-\tau^2}),\, c_2=(\half1{1-\tau^2},\half{\tau}{1-\tau^2}),c_3=(1,0)\},
$$
and the nontrivial part of the spectrum is $$\lambda_1=\lambda_2=\frac{1+\tau^2}{2(1-\tau^2)}, \qquad \lambda_3=\half12(1-\tau^2)<\frac12.
$$
It follows that
\begin{itemize}
\item
if $0<\tau<1$ then $\sgn(a_1,a_2,a_3)=(-,-,+)$, hence $H(\tau)$ has type (ii);
\smallskip
\item
if $\tau>1$ then $\sgn(a_1,a_2,a_3)=(+,+,+)$, hence $H(\tau)$ has type (i).
\end{itemize}

\smallskip\smallskip

\end{example}

\begin{example}[Type (ii)]\label{ex:typeiii}
Let $A=\R{}\times \R{}$ be the direct product of two copies of reals $\R{}$. Then $\Idm(A)=\{c_0=0,c_1=(1,0),c_2=(0,1),c_3=(1,1)\}$, and $c_3$ is the algebra unit. In the above notation, the nontrivial part of the spectrum is $\lambda_1=\lambda_2=0$ and $\lambda_3=1$, hence using \eqref{aui} we obtain
$$
(a_0,a_1,a_2,a_3)=(-1,1,1,-1),
$$
therefore $A$ has type (iii). Note that the zero idempotent and the algebra unit have  negative sign, while the two adjacent idempotents $e_1=e_3-e_2$  have  positive sign.
\end{example}

We emphasize that any deformation (i.e. a continuous path in the variety of all nonassociative commutative algebras) of any generic algebra $A$ switching its  type must pass through non-generic algebras. \marginpar{......}

\section{Real generic algebras and Riccati type ODEs}

We finish this section by indicating some useful correspondence between the three algebra types and three possible phase portraits of quadratic systems of ODEs with two independent variables. Recall that given a system of quadratic ordinary differential equations
\begin{equation}\label{ODE1}
\frac{dX}{dt}=F(X),
\end{equation}
where $F:\R{n}\to \R{n}$ is a homogeneous  degree two map, one can associate a commutative  algebra $A_F$ on $\R{n}$ with multiplication
\begin{equation}\label{ODE2}
x\circ y=\half14(F(x+y)-F(x-y))
\end{equation}
(note that the right hand side is obviously bilinear in $x,y$).
In the introduced notation, \eqref{ODE1} becomes a Riccati type ODE
\begin{equation}\label{ODE3}
\dot{X}=X\circ X=X^2
\end{equation}
on the algebra $A_F$. This construction is well-known and there exists a nice correspondence between the standard algebraic concepts and ODE tools, we refer the interested reader to \cite{Markus}, \cite{Walcher94}, \cite{GrWalcher}, \cite{Krasnov09}, \cite{Krasnov13}.
For example, $c\in \Idm(A_F)$ if and only if the one-dimensional subspace spanned by $c$ is an integral curve of \eqref{ODE1} (i.e. the separate variable solution $X(t)=f(t)c$ is a solution to \eqref{ODE1}). In ODEs the idempotent is often called  a  Darboux point and the Peirce numbers are known as the Kovalevskaja exponents, \cite{Zhang}. There are many well-established quadratic ODEs, in particular, in mathematical biology or  dynamics of the so-called second order chemical reactions (i.e.
the reactions with a rate proportional to the concentration of the square of a single reactant or the product of the concentrations of two reactants) \cite{Aris}.

In what follows, we focus on the particular case $\dim A_F=n=2$. Recall that Berlinskii's Theorem \cite{Berlin60} classifies the  global behaviour of a quadratic ODE system in the plane, in particular, it describes all possible configurations of critical points. More precisely, if a plane ODE has four singular points in a finite part of the phase plane, then only one of the following cases is possible:
\begin{itemize}
\item[(B1)] three singular points are vertices of a triangle containing the fourth point inside and this point is a saddle while the others are antisaddles (a center, a focus or a node).

\item[(B2)] three singular points are vertices of a triangle containing the fourth point inside, and this point is a antisaddle while the others are saddles.
\item[(B3)] these points are vertices of a convex quadrangular, where two opposite vertices are saddles (antisaddles) and two others are antisaddles (saddles);

\end{itemize}

\begin{table}[th]
\begin{tabular}{ccccc}\renewcommand\arraystretch{1.5}
$\vphantom{\int\limits_{S^2}^{S^2}}F=(x^2-y^2,-2xy)$ && $F=(x^2-y^2,\frac23xy)$ && $sF=(x^2,y^2)$\\
$\vphantom{\int\limits_{S^2}^{S^2}}$(B1) && (B2) && (B3)\\

\includegraphics[height=3.3cm]{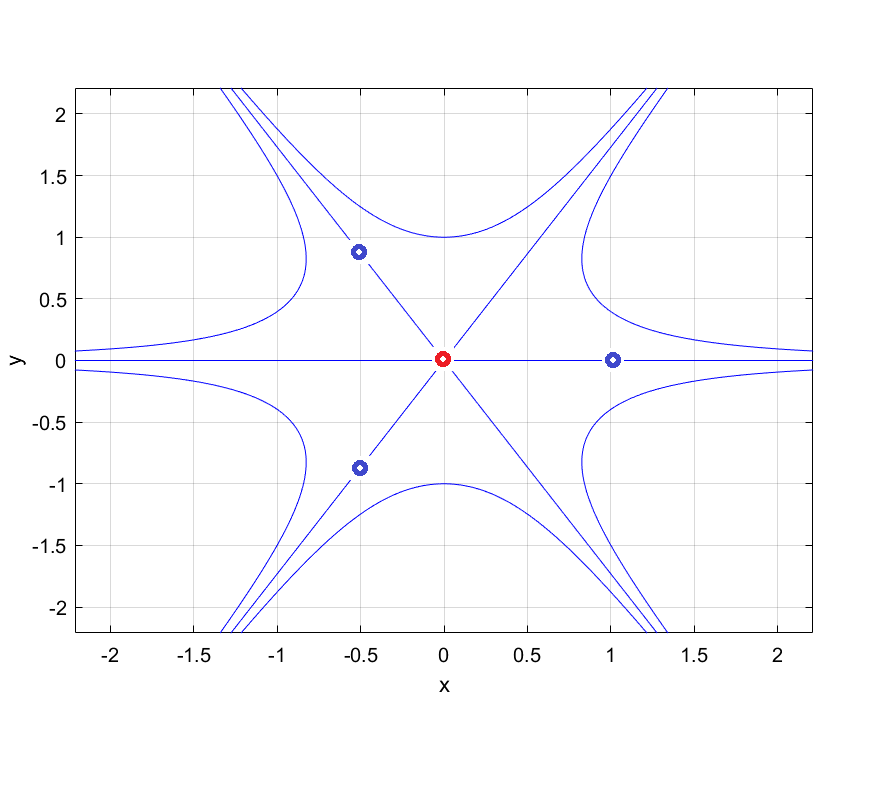}&&
\includegraphics[height=3.3cm]{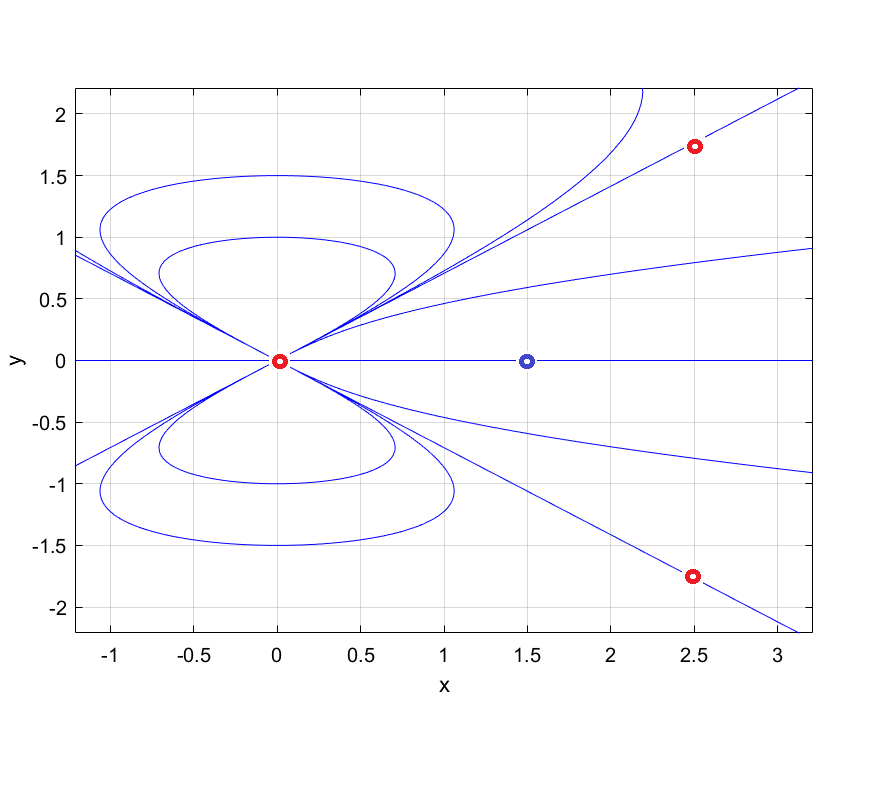}&&
\includegraphics[height=3.3cm]{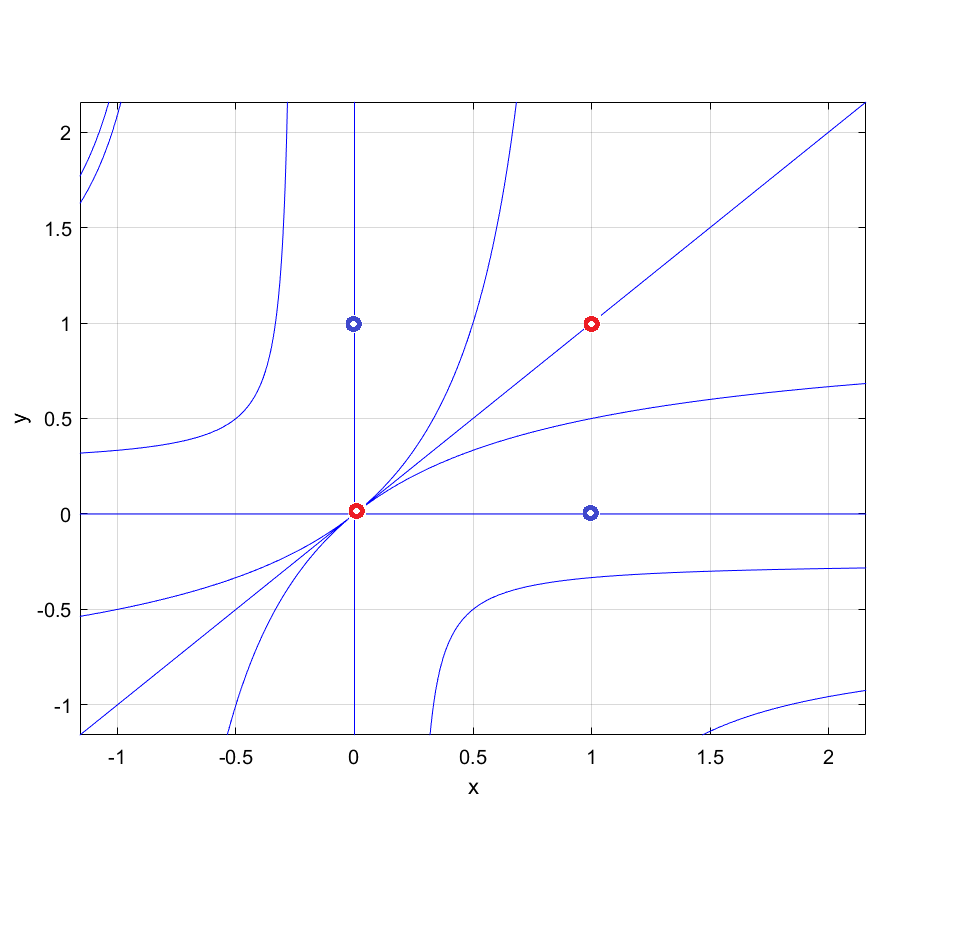}\\

type (i) && type (ii) && type (iii)\\
\\
$\vphantom{\int\limits_{S^2}}$$H(\sqrt{3})$ && $H(\half1{\sqrt{3}})$ && $\R{}\times \R{}$\\
\smallskip
\end{tabular}\label{tab1}
\caption{The ODE $\leftrightarrow$ real generic NA correspondence}
\end{table}

The reader easily recognize that these  cases correspond to exactly to the algebras of type (i), (ii) and (iii). Let us explain this in more detail in the case of the quadratic systems. To this end, let us consider \eqref{ODE3} in a two-dimensional real generic algebra $A$ with idempotents $\Idm(A)=\{c_0=0, c_1,c_2,c_3\}$. Then the rays
$$
\ell_{\pm}(c_i)=\{tc_i:t\in \R{\pm}\}, \quad i=1,2,3,
$$
split $A\cong \R{2}$ into 6 sector domains with vertices at the origin. Clearly, that each $\ell_{\pm}(c_i)$ is an integral curve of \eqref{ODE3}, thus each of the sector 6 domains is invariant under flow.  Then the three different types of plane quadratic ODEs corresponding to the three two-dimensional real generic algebra types are shown on Table~\ref{tab1}.

\section{Some final remarks and questions}

The case of real generic algebras of $\dim A\ge 3$ is more subtle.  It is interesting to know which configurations of idempotents are realizable and which geometrical concepts are naturally appear in this analysis. Recall that in the two-dimensional case considered in the previous section,   critical (non-generic) situations occur exactly when three idempotents fall on one line. The situation in three dimensions is more complicated and involves quadric surfaces instead. We are still able to get a nice description  here but this requires a more delicate work with the syzygies in the spirit of \cite{KT2018a}. Full details will appear elsewhere.

Note that any idempotent generates a one-dimensional subalgebra. Two-dimensional subalgebras of generic algebras appear naturally in the configurations considered above. This motivates a very natural question: what can be said about the possible number of subalgebras of a given generic algebra? For example, the direct algebra product $\R{}\times\R{}\times \R{}$ has $6$ two-dimensional subalgebras. Is there exists more than six subalgebras in dimension three? The latter example is a unital algebra. Is the unitality necessary for that?

\subsection*{Acknowledgment}
This work was done while the first author visited Link\"oping's University. He would like to thank the Mathematical Institution of Link\"oping's University for hospitality. The second author was partially supported by  G.S. Magnusons Foundation, grant MG 2017-0101.

\bibliographystyle{amsplain}

\def\cprime{$'$} \def\cprime{$'$}
\providecommand{\bysame}{\leavevmode\hbox to3em{\hrulefill}\thinspace}
\providecommand{\MR}{\relax\ifhmode\unskip\space\fi MR }
\providecommand{\MRhref}[2]{%
  \href{http://www.ams.org/mathscinet-getitem?mr=#1}{#2}
}
\providecommand{\href}[2]{#2}

\end{document}